\documentclass[11pt, reqno]{amsart}
\usepackage{amsfonts,amsmath,amssymb,amsthm,epsfig,cite,graphicx,hyperref,
	color, esint,fancyhdr, enumerate, latexsym,amsrefs, makecell}

    \usepackage{hyperref}
    \usepackage{cleveref}
\usepackage{color}
\usepackage[mathscr]{eucal}
\usepackage{comment}

\numberwithin{equation}{section}

\theoremstyle{definition}
\newtheorem{definition}{Definition}[section]
\newtheorem{example}[definition]{Example}

\theoremstyle{remark}
\newtheorem{remark}[definition]{Remark}
\theoremstyle{plain}
\newtheorem{theorem}[definition]{Theorem}
\newtheorem{lemma}[definition]{Lemma}

\newtheorem{result}[definition]{Result}

\newtheorem{corollary}[definition]{Corollary}

\newcommand{\eps}{\varepsilon}

\newcommand{\al}{\alpha}

\newcommand{\natu}{\mathbb{N}}
\newcommand{\vphi}{\varphi}

\newcommand{\Om}{\Omega}

\newcommand{\bdy}{\partial}

\newcommand{\smoo}{\mathcal{C}}

\newcommand{\rl}{{\sf Re}}

\newcommand{\impl}{\Longrightarrow}

\newcommand{\dist}{{\rm dist}}

\newcommand{\norm}[1]{\left\lVert #1\right\rVert}
\newcommand{\XO}{\mathfrak X_{\mathcal O}}

\newcommand{\CC}{\mathbb{C}^2}
\newcommand{\cplx}{\mathbb{C}}

\newcommand{\rea}{\mathbb{R}}
\newcommand{\cn}{\mathbb{C}^n}

\begin{document}

	\title[Flow-invariant Runge domains and global linearization]{Flow invariant Runge domains and global linearization of holomorphic vector fields }
	\author{Sanjoy Chatterjee and Sushil Gorai}
	\address{Department of Mathematics and Statistics, Indian  Institute of Technology Kanpur}
	\email{ramvivsar@gmail.com}
	
	\address{Department of Mathematics and Statistics, Indian Institute of Science Education and Research Kolkata,
		Mohanpur -- 741 246}
	\email{sushil.gorai@iiserkol.ac.in, sushil.gorai@gmail.com}

	\date{\today}


	\begin{abstract}
		In this paper we study two problems concerning holomorphic flows on $\mathbb C^n$.
First, we prove Runge-type results for positive-time flow invariant
domains. For a linear flow $e^{tA}$, where $A\in GL(n,\mathbb C)$,
let $E^s$, $E^u$, and $E^c$ denote the stable, unstable, and center
subspaces of $A$, respectively. We show that if a positive-time flow
invariant domain $\Omega\subset\mathbb C^n$ contains the origin and the
center subspace, and if $E^u\oplus E^c$ has positive distance from
$\bdy\Omega$, then $\Omega$ is a Runge domain. We also discuss
additional classes and constructions of flow invariant Runge domains
arising from holomorphic dynamics.

Second, we investigate the global linearization of holomorphic vector
fields by automorphisms of \(\mathbb C^n\). We prove that a complete
holomorphic vector field
$V$ on $\cplx^n$ with a globally attracting fixed point, satisfying certain integrability condition can be globally linearized by an automorphism of $\cplx^n$. As a corollary we obtain the global linearization of vector fields of the form
        $V(z)=Az+O(\|z\|^m)$
        near $z= 0$, under certain
spectral-gap condition. The conjugating automorphism is obtained as the
limit of the family $e^{-tA}X_t$, where $X_t$ is the flow of $V$. Some examples are provided for illustration.
	\end{abstract}	

	\maketitle
	
	
	\section{Introduction}\label{Intro}
A {\em holomorphic vector field} $V$ on a domain  $\Om \subset \mathbb{C}^{n}$ is a real vector field on $\Om$  such that  $$V(z)=\sum_{j=1}^{n} a_{j}(z) \frac{\partial }{\partial  x_{j}}+b_{j}(z) \frac{\partial }{\partial  y_{j}}$$ such that $(a_j+ib_{j})$ is holomorphic function on $\Om$  for all $j \in \{1,2, \cdots ,n\}$.  We identify a holomorphic vector field on a domain $\Omega\subset\cplx^n$ with a holomorphic map
\[
  V=(V_1,\ldots,V_n):\Omega\to\cplx^n,
\]
whose flow $X_t$ solves
\[
  \frac{d}{dt}X_t(z)=V(X_t(z)),\qquad X_0(z)=z.
\]
We denote the space of holomorphic vector fields on $\cplx^n$ by $\XO(\cplx^n)$.  Unless explicitly stated otherwise, a complete holomorphic vector field is understood to have a flow defined for all complex times; its restriction to $t\in\rea$ is the real-time flow.  A domain $\Omega\subset\cplx^n$ is called \emph{positive-time flow invariant} with respect to $V$ if
$X_t(\Omega)\subset \Omega$ for every $t\geq 0.$

The first part of the paper concerns the Runge property of such domains.  Recall that a domain $\Omega\subset\cplx^n$ is Runge if every holomorphic function on $\Omega$ can be approximated, uniformly on compact subsets of $\Omega$, by holomorphic polynomials.  Runge domains play a fundamental role in several complex variables.  In one complex variable, the Runge property is characterized by the absence of relatively compact connected components of $\cplx\setminus\Omega$.  In higher dimensions no comparably simple topological characterization is available, and deciding whether a given domain is Runge is often subtle.  One of our aims is to identify new classes of Runge domains in higher dimensions that arise naturally from holomorphic dynamics.
We begin with linear dynamics.  Let $A\in GL(n,\cplx)$ and consider the linear flow $e^{tA}$.  The spectrum of $A$ gives the standard decomposition
\[
  \cplx^n=E^s\oplus E^u\oplus E^c
\]
into the stable, unstable, and center subspaces; see Section~\ref{sec-prelims} for the precise definitions.  Put
\[
  E^{cu}:=E^u\oplus E^c,
\]
and let
\[
  P^s:\cplx^n\to E^s,
  \qquad
  P^{cu}:\cplx^n\to E^{cu}
\]
be the projections associated with the splitting $\cplx^n=E^s\oplus E^{cu}$.  For a set $S\subset\cplx^n$, we use
\[
  \dist(S,\partial\Omega):=\inf\{\norm{x-y}:x\in S,\ y\in\partial\Omega\}.
\]

\begin{theorem}\label{T:orbitwise-thickness}
Let \(\Omega\subset\mathbb C^n\) be a positive-time flow invariant domain 
with
$E^{cu}\subset\Omega$.
Suppose that for every compact set \(K\subset\mathbb C^n\), there exists
\(T_K>0\) such that
\[
        \operatorname{dist}
        \bigl(e^{tA}P^{cu}(K),\partial\Omega\bigr)
        >
        \sup_{z\in K}\|e^{tA}P^s z\|
        \quad \text{for all } t\geq T_K .
        \tag{3.1}\label{E:orbitwise-thickness}
\]
Then \(\Omega\) is a Runge domain.
\end{theorem}
\begin{remark}
The conclusion of Theorem~\ref{T:Runge} does not, in general, hold without the assumption \eqref{E:orbitwise-thickness}. A  counterexample is given in Section~\ref{sec-Runge}. 
\end{remark}

\begin{corollary}\label{T:Runge}
Let $\Om \subset \cn $ be a positive-time flow invariant domain with respect to $A$ and $0\in\Omega$. Suppose that \[
 E_{c} \subset \Om\quad \text{and}\quad \text{dist}(E_{u} \oplus E_{c},\partial \Om)>0.
 \]
 Then $\Om$ is a Runge domain. 
\end{corollary}

\noindent 
Next we focus at the complete holomorphic vector fields $V$ with a globally attracting fixed point at the origin, i.e., the flow $X_t(z)\to 0$ as $t\to \infty$. We also sometimes call such a vector field as complete holomorphic asymptotically stable with the origin as the equilibrium point.
Next, we give a short proof of a result from our preprint \cite{CG} which provides another broad class of Runge domains; it was used by
Forstneri\v c~\cite{FR2025} to generalize \cite[Theorem~1.5]{CG}.
\begin{theorem}\label{thm-spirallikeRunge1}
		Let $F$ be a complete holomorphic vector field with a globally attracting fixed point at the origin. Let $\Om \subset\mathbb{C}^n$ be a domain containing the origin and invariant under positive-time flows of $F$. Then $\Omega$ is a Runge domain.
	\end{theorem}
\begin{remark}
In the case of linear vector fields, Theorem~\ref{thm-spirallikeRunge1} says
that if all eigenvalues of the linear vector field have negative real parts,
then every positive times flow invariant domain with respect to that vector field is
Runge. It generalizes earlier results (for instance \cite[Theorem~3.1]{Hamada15}) substantially (see \cite{CG} for details).
\end{remark}

The second part of the paper concerns global linearization of holomorphic vector fields on $\cplx^n$.  Given a complex manifold $M$ and a holomorphic vector field $V$ on $M$, the associated flow, when defined, consists of biholomorphic maps.  The basic question is whether there exists a global change of coordinates, namely a biholomorphism $\Phi:M\to M$, that transforms the dynamics into a linear model.  Equivalently, one asks whether $V$ can be conjugated to a linear vector field so that, in global coordinates, its flow becomes a one-parameter group of linear transformations.

Such a question has genuine global obstructions.  On general complex manifolds, some linear structure must already be present.  We focus on $\cplx^n$.  Even in $\cplx^n$, $n\geq 2$, not every complete holomorphic vector field is linearizable; see Suzuki's work on holomorphic $\cplx$-actions on $\cplx^2$ \cite{Suzuki1970}.  Recall that two holomorphic vector fields $V_j\in\XO(\Omega_j)$, $j=1,2$, with flows $X_t^j$, are \emph{conjugate} if there exists a biholomorphism $H:\Omega_1\to\Omega_2$ such that
\[
  H\circ X_t^1\circ H^{-1}=X_t^2
\]
for every time $t$ for which both sides are defined.

This problem is a global counterpart of the classical local linearization theory of Poincar\'e, Siegel, Brjuno, Sternberg, and others.  In the local setting one studies a vector field near an equilibrium point and seeks a holomorphic coordinate change that reduces it to its linear part.  Under suitable spectral hypotheses, such a reduction is possible.  We recall one classical theorem from \cite[page 190]{Arnold1988}.



\begin{result}[Poincar\'{e}]\label{R:Poincare}
Let $V \in \mathfrak{X}_{\mathcal{O}}(\Omega)$ be a holomorphic vector field such that $V(a)=0$, and let $A = DV(a)$ denote its linearization at $a$ for some $a \in \Om$ and $\sigma(A)$ be the spectrum of $A$. Assume that 
\begin{itemize}
\item [(i)]$A$ is non-resonant, that is,
\begin{align}\label{eq:nonresonance}
\lambda_{j} \neq \sum_{l=1}^{n} m_{l}\lambda_{l},
\end{align}
for all $j \in \{1,2,\dots,n\}$ and all $n$-tuples $(m_{1},\dots,m_{n})$ of non-negative integers with $\sum_{l=1}^{n} m_{l} \geq 2$, where $\lambda_{j} \in \sigma(A)$; and
\item[(ii)] the origin does not lie in the convex hull of $\lambda_1,\dots, \lambda_n$ in $\cplx$.
\end{itemize}
Then  $V$ is locally conjugate to its linearization via a biholomorphism defined in a neighborhood of the equilibrium point $a$.
\end{result}

Sternberg's examples show that, without a non-resonance condition such as \eqref{eq:nonresonance}, a holomorphic vector field need not be locally linearizable by a biholomorphic change of coordinates; see \cite[p.~127]{perko}.  For further background on local analytic differential equations, see \cite{LECTURENOTEINODE} and the references therein.

Global linearization is more rigid.  One must control not only the local behavior near a fixed point, but also the geometry of the ambient space and the global structure of the flow.  Completeness is necessary for a globally defined flow, but it is not sufficient.  Global invariants such as orbit structure, holomorphic first integrals, and invariant foliations may obstruct linearization.

We study global linearization under asymptotic stability assumptions.  Our first criterion is formulated in terms of an explicit improper integral.

 \begin{theorem}
\label{T:integrability-linearization}
Let \(V\in\mathcal X_{\mathcal O}(\mathbb C^n)\) be a complete
holomorphic vector field whose origin is a globally asymptotically
stable fixed point and $A=DV(0)$. Let \(X_t\) denote its flow, and
$ V(z)=Az+R(z)$ where
$ R(z)=o(\|z\|) $ as $ z\to 0$.
\noindent
Suppose that, for every compact set \(K\subset\mathbb C^n\), the improper
integral
\[
        \int_0^\infty e^{-sA}R(X_s(z))\,ds
\]
converges uniformly for \(z\in K\). 
Then the limit
\[
        F(z):=\lim_{t\to\infty}e^{-tA}X_t(z)
\]
exists locally uniformly on \(\mathbb C^n\), and \(F\in
\operatorname{Aut}(\mathbb C^n)\). Moreover,
\[
        F\circ X_t\circ F^{-1}=e^{tA}
        \qquad \text{for all } t\in\mathbb \cplx.
\]

\end{theorem}

\noindent The next corollary formulated in terms of a spectral gap condition. For that we set
\[
        k_-(A):=\min_{\lambda\in\sigma(A)}\operatorname{Re}\lambda,
        \qquad
        k_+(A):=\max_{\lambda\in\sigma(A)}\operatorname{Re}\lambda .
\]
 
\begin{corollary}\label{L:composconver}
Let $V \in \mathfrak{X}_{\mathcal{O}}{(\mathbb{C}^{n})}$ $(n \geq 2)$ be a complete, globally asymptotically stable  vector field  and $V$ is of the form  
\[
V(z)=Az+O(|z|^m)\quad\;\text{as}\quad \|z\| \to 0.
\]
Suppose that $mk_+(A)<k_-(A)$. Then there exists $F\in {\rm Aut}(\cn)$ such that $F\circ X_{t} \circ F^{-1}(z)=e^{tA}z$  for all $z \in \cn$ for all  $t \in \cplx$.
 \end{corollary}

The conclusion of Theorem\ref{T:integrability-linearization} and Corollary~\ref{L:composconver} can often be obtained by first proving a local linearization theorem and then extending the local conjugacy along the attracting flow.  The point of the present approach is that the global linearizing automorphism is given explicitly by the limit of $e^{-tA}X_t$.  This limit construction also has applications to approximation by automorphisms on certain flow-invariant domains; see \cite{CG}.  The following elementary consequence illustrates how invariant objects are transferred from the linear model to the nonlinear flow.

\begin{corollary}\label{coro-transfer}
Let \(V\) satisfy the hypotheses of Theorem~\ref{T:integrability-linearization}, and let
\(F\in\operatorname{Aut}(\mathbb C^n)\) be the linearizing automorphism, so that
\[
        F\circ X_t\circ F^{-1}=e^{tA}.
\]
Then the following hold.

\begin{itemize}
\item If \(D\subset\mathbb C^n\) satisfies
$e^{tA}(D)\subset D$ for all $t\geq 0$,
then
$\Omega:=F^{-1}(D)$ is a positive-time flow invariant domain for $V$.


\item If \(J\) is a first integral of the linear flow, that is,
\[
        J(e^{tA}w)=J(w),
\]
then
$I:=J\circ F$
is a first integral of the nonlinear flow:
\[
        I(X_t(z))=I(z).
\]
\end{itemize}
\end{corollary}


\section{Technical Preliminaries}\label{sec-prelims}
In this section we collect a few standard facts used in the proofs.  We begin with the stable, unstable, and center subspaces of a matrix.
	
	\begin{definition}
	    Let $A \in M_{n}(\cplx)$ and $\sigma(A)=\{\lambda_{1}, \lambda_{2}, \cdots ,\lambda_{m}\}$. Let $m_{j}$ denote the size of the largest Jordan block of the matrix $A$ corresponding to $\lambda_{j}$. The stable, unstable, and center subspaces of the matrix $A$ are 
 \begin{align*}
     E_{s}&:= \bigoplus_{\rl(\lambda_{j})<0}\text{Ker}((A-\lambda_{j}I)^{m_{j}})\\
      E_{u}&:= \bigoplus_{\rl(\lambda_{j})>0}\text{Ker}((A-\lambda_{j}I)^{m_{j}})\\
       E_{c}&:= \bigoplus_{\rl(\lambda_{j})=0}\text{Ker}((A-\lambda_{j}I)^{m_{j}}).
 \end{align*}


	\end{definition}

We shall use the following elementary matrix estimates.
\begin{result} \cite{SimDav}\label{L:exp}
		 For every $\epsilon>0$ there exist $\rho(\epsilon), m(\epsilon)>0$, such that 
		 \begin{align*}
		 	\|e^{-tA}z\| &\leq \rho(\epsilon)e^{-(k_{-}(A)-\epsilon)t}\|z\|\\
		 	\|e^{tA}z\| &\leq m(\epsilon)e^{(k_{+}(A)+\epsilon)t}\|z\|
		 \end{align*}
	for all $t \geq 0$, and for all $z \in \cn$. 
\end{result}
 
The next result is a vector-valued version of Rouch\'e's theorem due to Lloyd \cite[Theorem 3]{Lloyd}. This will be used in the proofs of global linearization theorems.
 \begin{result}\cite[Theorem 3]{Lloyd}\label{R:Lloyd}
Let $D\subset\cplx^n$ be a bounded domain.  Suppose $f,g:\overline D\to\cplx^n$ are continuous and holomorphic on $D$, and assume that
\[
  \norm{g(z)}<\norm{f(z)}
  \qquad\text{for every } z\in\partial D.
\]
Then $f$ and $f+g$ have the same number of zeros in $D$, counted with multiplicities.
 \end{result}

\section{Runge domains}\label{sec-Runge}

In this section we prove Theorems~\ref{T:orbitwise-thickness}, \ref{thm-spirallikeRunge1} and Corollary~\ref{T:Runge}. Then we give examples and constructions of Runge domains arising from holomorphic flows. We begin by establishing a couple of preparatory lemmas.

\begin{lemma}\label{L: unstable-contained}
		Let $\Om  \subset \cn$ be a flow invariant domain with respect to $A \in GL(n, \mathbb{C})$ containing the origin. Then the unstable subspace $E_{u}$ of the matrix $A$ is contained in $\Om$.
	\end{lemma}
	\begin{proof}
		Let $E_{u}$ be the unstable subspace of $A$ and $v \in E_{u}$. Since $e^{-tA}v \to 0$ as $t \to \infty$,  there exist $T>0$ such that $e^{-TA}v \in \Om$. Note that  $\Om$ is a flow invariant  domain with respect to $A$. Therefore,   $e^{TA}(e^{-TA}v) =v\in \Om$.
\end{proof}

	\begin{lemma}\label{P: global asymptotic property}
		Let $V \in \mathfrak{X}_{\mathcal{O}}{(\mathbb{C}^{n})}$ be a complete vector field with a globally attracting fixed point at the origin and $\Omega$ be a positive-time flow invariant domain for $V$ containing the origin. Then, for every  compact subset $K \subset \mathbb{C}^{n}$, there exists a real number $t_{K}>0$ such that $X(t,z) \in \Omega$ for all $t >t_{K}$ and for all $z \in K$.
	\end{lemma}
	\begin{proof}
		Let $r>0$ such that $B(0,r) \subset \Om$.  Since the origin is the  globally attracting fixed point of  $F$, there exists $\delta>0$ such that $X(t,z) \in B(0,r)$ $\forall t \geq 0$ and  $\forall z \in B(0, \delta)$.
  Note that $\lim_{t \to \infty}X(t,z) \to 0$ as  $t \to \infty $ for every  $z \in \cn$. Therefore, for every $z \in K$  there exist $t_{z}>0$ such that $X(t_{z},z) \in B(0, \frac{\delta}{3})$. Hence, from the choice of $\delta>0$, we obtain 
  $$
  X(t,z) \in B(0,r) \subseteq \Omega\;\; \forall t \geq  t_{z}.
  $$
  Let $K \subset \cn$ be a compact subset. For every $z \in K$ there exist $r_{z}, t_{z}>0$ such that $X(t,w) \in B(0, \frac{\delta}{3})$ for all $w \in B(z,r_{z})$ and $t\geq t_{z}$. Since $K$ is compact, hence, there exists $m \in \mathbb{N}$ and $\{z_{1},z_{2}, \ldots ,z_{m}\} \subset K$ such that $K \subset \cup_{j=1}^{m}B(z_{j},r_{z_{j}})$. Let $t_{K}=\max \{t_{z_{1}},t_{z_{2}}, \ldots ,t_{z_{m}}\}$. Therefore, $X(t,z) \in B(0, r)$ for all $z \in K$  for all $t>t_{K}$.
  
\end{proof}

\begin{proof}[Proof of Theorem~\ref{T:orbitwise-thickness}]
We divide the proof into two steps.

First, we prove that for every compact set \(K\subset\mathbb C^n\),
there exists \(T_K>0\) such that
\[
        e^{tA}K\subset\Omega
        \qquad \text{for all } t\geq T_K .
        \tag{3.2}\label{E:absorbing}
\]
Let \(K\subset\mathbb C^n\) be compact. By hypothesis, there exists
\(T_K>0\) such that \eqref{E:orbitwise-thickness} holds for every
\(t\geq T_K\). Fix \(z\in K\) and \(t\geq T_K\). Write
\[
        z=P^s z+P^{cu}z .
\]
Since \(E^s\) and \(E^{cu}\) are invariant under \(A\), we have
\[
        e^{tA}z
        =
        e^{tA}P^s z+e^{tA}P^{cu}z .
\]

Note that 
$a_{t}(z):= e^{tA}P^{cu}z\in E^{cu}\subset\Omega$, also $ a_{t}(z)\in e^{tA}P^{cu}(K).$
Therefore,
\[
\operatorname{dist}(a_t(z),\partial\Omega)
        \geq
        \operatorname{dist}
        \bigl(e^{tA}P^{cu}(K),\partial\Omega\bigr).
\]
By \eqref{E:orbitwise-thickness}, we get
\[
        \operatorname{dist}(a_t(z),\partial\Omega)
        >
        \sup_{\zeta\in K}\|e^{tA}P^s\zeta\|
        \geq
        \|e^{tA}P^s z\|.
\]
Hence
\[
        \|e^{tA}P^s z\|
        <
        \operatorname{dist}(a_t(z),\partial\Omega).
\]
We now use the elementary fact that if \(a\in\Omega\) and
$\|h\|<\operatorname{dist}(a,\partial\Omega)$,
then \(a+h\in\Omega\). Indeed, if \(a+h\notin\Omega\), then the line
segment joining \(a\) to \(a+h\) must meet \(\partial\Omega\). Thus
there would exist \(0\leq s\leq 1\) such that
$
        a+sh\in\partial\Omega,
$
and consequently
\[
        \operatorname{dist}(a,\partial\Omega)\leq \|sh\|\leq \|h\|,
\]
which is a contradiction.
Applying this with
$
        a=e^{tA}P^{cu}z,
        $ and  $
        h=e^{tA}P^s z,
$
we obtain
\[
        e^{tA}z
        =
        e^{tA}P^{cu}z+e^{tA}P^s z
        \in\Omega.
\]
Since \(z\in K\) was arbitrary, this proves
$e^{tA}K\subset\Omega$ for all $t\geq T_K$.

We now prove that \(\Omega\) is Runge. Let \(L\Subset\Omega\) be compact,
let \(f\in\mathcal O(\Omega)\), and let \(\mu\) be a complex Borel measure
supported on \(L\) such that
\[
        \int_L p\,d\mu=0
\]
for every holomorphic polynomial \(p\in\mathbb C[z_1,\ldots,z_n]\).
By the standard duality criterion for Runge domains, it is enough to
prove that
\[
        \int_L f\,d\mu=0 .
\]
Define
\[
        S_L
        :=
        \{\tau\in\mathbb C:e^{\tau A}L\subset\Omega\}.
\]
Since the map
$(\tau,z)\longmapsto e^{\tau A}z$
is holomorphic, \(S_L\) is open. Moreover, since \(L\subset\Omega\) and
\(\Omega\) is positive-time invariant domain under \(e^{tA}\), we have
$[0,\infty)\subset S_L$.
Define
\[
        H(\tau)
        :=
        \int_L f(e^{\tau A}z)\,d\mu(z),
        \qquad \tau\in S_L .
\]
Then \(H\) is holomorphic on \(S_L\).
Choose \(R>0\) such that
$L\subset B(0,R)$.
By Equation~\eqref{E:absorbing} proved above, applied to the compact set
\(\overline{B(0,R)}\), there exists \(T_R>0\) such that
\[
        e^{tA}\overline{B(0,R)}\subset\Omega
        \qquad \text{for all } t\geq T_R .
\]
Therefore, for every \(t\geq T_R\), the function
$g(z):=f(e^{tA}z)$
is holomorphic on a neighborhood of \(\overline{B(0,R)}\). Hence, there exists a sequence
of holomorphic polynomials \(\{p_j\}\) such that $p_j\longrightarrow g$
uniformly on \(L\). Hence
\[
\begin{aligned}
        H(t)
        &
        =\int_L f(e^{tA}z)\,d\mu(z)  
        =\int_L g(z)\,d\mu(z)
        =
        \lim_{j\to\infty}\int_L p_j(z)\,d\mu(z) 
        =
        0
\end{aligned}
\]
for every \(t\geq T_R\).

Thus, \(H\) vanishes on the real interval \([T_R,\infty)\). By the
identity theorem, \(H\) vanishes on the connected component of \(S_L\)
which contains \([T_R,\infty)\). Since
        $[0,\infty)\subset S_L$,
the point \(0\) lies in the same connected component. Therefore
\[
        0=H(0)=\int_L f(z)\,d\mu(z).
\]
This proves that \(\Omega\) is Runge.
\end{proof}

\begin{proof}[Proof of Corollary~\ref{T:Runge}]
By Lemma~\ref{L: unstable-contained}, $E^u\subset\Omega$.  Since $E^c\subset\Omega$ by assumption, we have $E^{cu}=E^u\oplus E^c\subset\Omega$.  Put
\[
  d:=\dist(E^{cu},\partial\Omega)>0.
\]
For every compact set $K\subset\cplx^n$,
\[
  \dist\bigl(e^{tA}P^{cu}(K),\partial\Omega\bigr)\ge d
  \quad \text{for every } t\ge 0,
\]
because $e^{tA}P^{cu}(K)\subset E^{cu}$.  On the other hand,
\[
  \sup_{z\in K}\norm{e^{tA}P^s z}\to 0
  \qquad (t\to\infty),
\]
since $P^s(K)$ is compact in the stable subspace.  Hence, Equation \eqref{E:orbitwise-thickness} holds for all sufficiently large $t$, and Theorem~\ref{T:orbitwise-thickness} applies.
\end{proof}

\smallskip
We now present the proof of Theorem~\ref{thm-spirallikeRunge1}.
    
 \begin{proof}[Proof of Theorem~\ref{thm-spirallikeRunge1}]
Let \(K\Subset\Omega\), let \(f\in\mathcal O(\Omega)\), and let \(\mu\) be
a complex Borel measure supported on \(K\) such that
\[
        \int_K p\,d\mu=0
\]
for every holomorphic polynomial \(p\in\mathbb C[z_1,\ldots,z_n]\). It is
enough to show that
\[
        \int_K f\,d\mu=0 .
\]
Let \(X_\tau\) denote the complex-time flow of \(F\). Since \(F\) is
complete, \(X_\tau\) is defined for every \(\tau\in\mathbb C\). Set
\[
        S_K:=\{\tau\in\mathbb C: X_\tau(K)\subset\Omega\}.
\]
Then \(S_K\) is open and contains the real half-line \([0,\infty)\). Define
\[
        H(\tau):=\int_K f(X_\tau(z))\,d\mu(z),
        \qquad \tau\in S_K .
\]
Then \(H\) is holomorphic on \(S_K\).

Choose \(R>0\) such that \(K\Subset B(0,R)\). By
Lemma~\ref{P: global asymptotic property}, there exists \(T>0\) such that
\[
        X_t(\overline{B(0,R)})\subset\Omega
        \qquad \text{for all } t\geq T .
\]
Thus, for every \(t\geq T\), the function
\[
        z\longmapsto f(X_t(z))
\]
is holomorphic on \(B(0,R)\). It can therefore be approximated uniformly
on \(K\) by holomorphic polynomials. Since \(\mu\) annihilates all
holomorphic polynomials, we obtain
\[
        H(t)=0
        \qquad \text{for all } t\geq T .
\]
By the identity theorem, \(H\) vanishes on the connected component of
\(S_K\) containing \([0,\infty)\). In particular,
\[
        \int_K f(z)\,d\mu(z)=H(0)=0 .
\]
Therefore \(\Omega\) is a Runge domain.
\end{proof}

\subsection{Construction of Runge Domains:}
In this subsection we give several constructions of Runge domains.
\begin{enumerate}
    \item \emph{Hyperbolic linear flows.}
Let
\(A\in GL(n,\mathbb C)\) be such that
$ \rl\lambda\neq 0
        \qquad \text{for every } \lambda\in\sigma(A).
$
Then
$
        \mathbb C^n=E^s\oplus E^u.
$
Let \(P^s:\mathbb C^n\to E^s\) denote the projection associated with this
splitting. Define a norm \(q\) on \(E^s\) by
\[
        q(v):=\sup_{t\geq 0}\|e^{tA}v\|,
        \qquad v\in E^s .
\]
This is finite because \(e^{tA}v\to 0\) exponentially on \(E^s\). Moreover,
\[
        q(e^{tA}v)\leq q(v)
        \qquad \text{for all } t\geq 0 .
\]
For \(r>0\), set
\[
        \Omega_r:=\{z\in\mathbb C^n:q(P^s z)<r\}.
\]
Then \(\Omega_r\) is a domain, \(E^u\subset\Omega_r\), and
\[
        e^{tA}(\Omega_r)\subset\Omega_r
        \qquad \text{for all } t\geq 0 .
\]
Furthermore, \(\operatorname{dist}(E^u,\partial\Omega_r)>0\). Indeed, since
\(q\) is equivalent to the Euclidean norm on \(E^s\), there is a constant
\(C>0\) such that $q(v)\leq C\|v\|$
        for all  $v\in E^s$.
If \(z\in\partial\Omega_r\), then \(q(P^s z)=r\), and hence
\[
        \|P^s z\|\geq r/C.
\]
For every \(w\in E^u\),
\[
        \|z-w\|\geq \frac{\|P^s z\|}{\|P^s\|}
        \geq \frac{r}{C\|P^s\|}.
\]
Thus \(\operatorname{dist}(E^u,\partial\Omega_r)>0\), and
Theorem~\ref{T:Runge} implies that \(\Omega_r\) is Runge.

\medskip

\item \emph{Lyapunov-type construction.} Let $V \in \mathfrak{X}_{\mathcal{O}}(\cn)$  be a complete holomorphic vector field such that the origin is the globally attractive equilibrium point of the vector field and $X(t,z)$ be its flow. Consider the function $\varphi: \cn \to 
 \mathbb{R}$ defined by \begin{align}\label{E:improperintegral}
\varphi(z):=\int_{0}^{\infty } \|X(t,z)\|\,dt.
 \end{align} 
Since all the eigenvalues of the matrix $DV(0)$ have strictly negative real parts, hence, by the Hartman-Grobman theorem we get that there exists a neighborhood $U$ of the origin and a  $\smoo^{1}$ smooth diffeomorphism $H:U \to H(U)$ such that $X(t,z)=H\circ e^{tA}\circ H^{-1}(z)$  $t \geq 0$ for all $z \in U$. Here the origin is the globally attractive equilibrium point of the vector field $V$. Therefore, all the eigenvalues of $A$ have strictly negative real parts. Consequently, $H(0)=0$. By our assumption,  there exists $t_{z}>0$ such that $X(t,z) \in U$ for all $t \geq t_{z}$. Therefore, we have 
\begin{align*}
  \int_{0}^{\infty } \|X(t,z)\|\,dt& =\int_{0}^{t_{z} } \|X(t,z)\|\,dt +\int_{t_{z}}^{\infty}\|X(t,z)\|\,dt\\
  &=\int_{0}^{t_{z} } \|X(t,z)\|\,dt+\int_{0}^{\infty}\|X(t+t_{z},z)\|\,dt.\end{align*}
Since $X(t_{z},z) \in U$, we have that 
 $X(t+t_{z},z)=H\circ e^{tA}\circ H^{-1}(X(t_{z},z))$. Here for  all $t \geq 0$ the map $H\circ e^{tA} \circ H^{-1}: U \to \cn$ smooth map. 
Therefore, using  the mean value inequality, for this map we get  that there exists $z_{0} \in U $  such that the following holds for all $t\geq 0$: 
\begin{align*}
    &\|H\circ e^{tA} \circ H^{-1}(X(t_{z},z))-(H\circ e^{tA} \circ H^{-1}(0))\|\\
    &\leq \|D(H\circ e^{tA} \circ H^{-1})(z_{0})\|\|X(t_{z},z)\|\\
    &\leq \|D H(e^{tA}H^{-1}z_{0})\|.\|e^{tA}\|.\|DH^{-1}(z_{0})\|\|X(t_{z},z)\|.
\end{align*}
Clearly, it $\|DH^{-1}(z_{0})\|\|X(t_{z},z)\|$  is bounded above by some positive constant. Since $\{e^{tA}H^{-1}z_{0}: t\geq 0\}$ lies in a compact subset, $\|DH(e^{tA}H^{-1}z_{0})\|$ is also bounded. We now  use  Result~\ref{L:exp} to conclude that the improper integral $$\int_{0}^{\infty}\|X(t+t_{z},z)\|\,dt <\infty. $$  Consequently, we get that  \eqref{E:improperintegral} converges. For every $c>0$ the domain $\Om_{c}:=\{z \in \cn: \varphi(z)<c\}$ are invariant with respect to the flow of the vector field $V.$ Clearly, for $z \in \cn$ and $s\geq 0$ we have $$\varphi(X(s,z))=\int_{0}^{\infty} \|X(t+s,z)\|\,dt=\int_{s}^{\infty}\|X(\tau, z)\|\,d\tau.$$ 
  Hence, $\varphi(X(s,z))+\int_{0}^{s}\|X(\tau, z)\|\,d\tau=\varphi(z)$. Therefore, if $z \in \Om_{c}$ then $\varphi(X(s,z))\leq \varphi(z)<c$.

\medskip

\item \emph{Flow tubes.} 
 Let \(V\) and \(X_t\) be as
above, and fix \(r>0\). Set
\[
        \Omega:=\bigcup_{t\geq 0}X_t(B(0,r)).
\]
Then \(\Omega\) is open, connected, contains the origin, and is positively
flow invariant. Moreover, by global asymptotic stability, there exists
\(T>0\) such that
\[
        X_t(\overline{B(0,r)})\subset B(0,r)
        \qquad \text{for all } t\geq T .
\]
Hence
\[
        \Omega
        =
        \bigcup_{0\leq t\leq T}X_t(B(0,r)),
\]
and in particular \(\Omega\neq\mathbb C^n\). By
Corollary~\ref{thm-spirallikeRunge1}, this domain is Runge.

\end{enumerate}
\medskip

We now give some examples of Runge domains arising from our theorems. Example~\ref{Ex:z1z2} demonstrates a Runge domain whose defining function is not plurisubharmonic.
\begin{example}\label{Ex:z1z2}
Let
\[
        \Omega:=\{(z,w)\in\mathbb C^2:|z|^2-|w|^3<1\}.
\]
Then \(\Omega\) is Runge. For a proof let
\[
        A=
        \begin{pmatrix}
        -1 & 0\\
        0 & 1
        \end{pmatrix}.
\]
The associated flow is
\[
        e^{tA}(z,w)=(e^{-t}z,e^t w).
\]
It is immediate that \(e^{tA}(\Omega)\subset\Omega\) for all \(t\geq 0\).
Moreover,
\[
        E^u=\{(0,w):w\in\mathbb C\}\subset\Omega.
\]
If \((z,w)\in\partial\Omega\), then
\[
        |z|^2-|w|^3=1,
\]
and hence \(|z|\geq 1\). Therefore every point of \(\partial\Omega\) has
distance at least \(1\) from \(E^u\). Thus
\[
        \operatorname{dist}(E^u,\partial\Omega)>0,
\]
and Corollary~\ref{T:Runge} implies that \(\Omega\) is Runge. Note that the function $\rho(z,w)=|z|^2-|w|^3-1$ is not psh on $\CC.$ Indeed, $\mathcal{L}_{\rho}((z,w),(v_{1},v_{2}))=|v_{1}|^2-\frac{9}{4}|w||v_{2}|^2$. Consequently, $\mathcal{L}_{\rho}((z,w),(0,v_{2}))<0$ for all $w\neq 0$ and $v_{2} \neq 0$.
\end{example}

Next we provide a class of Runge domains in $\cplx^n$ generalizing Example~\ref{Ex:z1z2} and having nontrivial centre subspace.
\begin{example}\label{Ex:general-family}
Let \(\Omega\subset\mathbb C^{n_1+n_2+n_3}\) be defined by
\[
\Omega
=
\left\{
(Z_1,Z_2,Z_3):
e^{-\|Z_3\|}
\left(
\sum_{j=1}^{n_1}|z_{1j}|^{m_j}
-
\sum_{j=1}^{n_2}|z_{2j}|^{\ell_j}
\right)
<1
\right\},
\]
where $Z_k=(z_{k1},\ldots,z_{kn_k})$ for $ k=1,2,3,$
and \(m_j,\ell_j\in\mathbb N\). Then \(\Omega\) is Runge.
Indeed, let
\[
        A=
        \begin{pmatrix}
        -I_{n_1} & 0 & 0\\
        0 & I_{n_2} & 0\\
        0 & 0 & iI_{n_3}
        \end{pmatrix}.
\]
Then
\[
        E^u\oplus E^c
        =
        \{(0,X,Y):X\in\mathbb C^{n_2},\ Y\in\mathbb C^{n_3}\}.
\]
Clearly, \(E^u\oplus E^c\subset\Omega\). If
\((Z_1,Z_2,Z_3)\in\Omega\), then
\[
        e^{tA}(Z_1,Z_2,Z_3)
        =
        (e^{-t}Z_1,e^tZ_2,e^{it}Z_3).
\]
Since
\[
\begin{aligned}
& e^{-\|e^{it}Z_3\|}
\left(
\sum_{j=1}^{n_1}|e^{-t}z_{1j}|^{m_j}
-
\sum_{j=1}^{n_2}|e^t z_{2j}|^{\ell_j}
\right)  \\
&\qquad \leq
e^{-\|Z_3\|}
\left(
\sum_{j=1}^{n_1}|z_{1j}|^{m_j}
-
\sum_{j=1}^{n_2}|z_{2j}|^{\ell_j}
\right)
<1,
\end{aligned}
\]
we have \(e^{tA}(\Omega)\subset\Omega\) for all \(t\geq 0\).

It remains to verify the distance condition. Suppose, to the contrary, that
\[
        \operatorname{dist}(E^u\oplus E^c,\partial\Omega)=0.
\]
Then there exist
\[
        (0,Z_2^\nu,Z_3^\nu)\in E^u\oplus E^c
\]
and
\[
        (W_1^\nu,W_2^\nu,W_3^\nu)\in\partial\Omega
\]
such that
\[
        \|(0,Z_2^\nu,Z_3^\nu)-(W_1^\nu,W_2^\nu,W_3^\nu)\|\to 0.
\]
In particular, \(W_1^\nu\to 0\). Since
\((W_1^\nu,W_2^\nu,W_3^\nu)\in\partial\Omega\), we have
\[
e^{-\|W_3^\nu\|}
\left(
\sum_{j=1}^{n_1}|w_{1j}^\nu|^{m_j}
-
\sum_{j=1}^{n_2}|w_{2j}^\nu|^{\ell_j}
\right)
=1.
\]
However, the left-hand side is bounded above by
\[
        \sum_{j=1}^{n_1}|w_{1j}^\nu|^{m_j},
\]
which tends to \(0\). This contradiction proves the distance condition.
Therefore \(\Omega\) is Runge by Theorem~\ref{T:Runge}.
\end{example}

The next example demonstrates that Theorem~\ref{T:orbitwise-thickness} is stronger than Corollary~\ref{T:Runge}.
\begin{example}
\label{Ex:orbitwise-thickness-not-starlike}
Let
\[
        A=
        \begin{pmatrix}
        -2 & 0\\
        0 & 1
        \end{pmatrix}.
\]
Then
\[
        E^s=\mathbb C\times\{0\},
        \qquad
        E^u=\{0\}\times\mathbb C,
        \qquad
        E^c=\{0\}.
\]
Define a continuous increasing function \(m:[0,\infty)\to(0,\infty)\)
by
\[
        m(r)=
        \begin{cases}
        1, & 0\leq r\leq 1,\\[4pt]
        1+9(r-1), & 1<r<2,\\[4pt]
        10, & r\geq 2.
        \end{cases}
\]
Set
\[
        \rho(r):=\frac{m(r)}{1+r},
        \qquad r\geq 0,
\]
and define
\[
        \Omega
        :=
        \left\{
        (z,w)\in\mathbb C^2:
        |z|<\rho(|w|)
        \right\}.
\]
We will show that \(\Omega\) satisfies conditions of Theorem~\ref{T:orbitwise-thickness}.
Clearly, \(E^{cu}\subset\Omega\).
The flow of \(A\) is
$(z,w)\mapsto (e^{-2t}z,e^t w).$
We claim that \(\Omega\) is positive-time invariant under this flow. Let
\((z,w)\in\Omega\), and put \(r=|w|\). Since \(m\) is increasing, we have
\[
        m(e^t r)\geq m(r).
\]
Also,
\[
        1+e^t r\leq e^t(1+r).
\]
Hence
\[
        \rho(e^t r)
        =
        \frac{m(e^t r)}{1+e^t r}
        \geq
        \frac{m(r)}{e^t(1+r)}
        =
        e^{-t}\rho(r).
\]
Therefore
\[
        |e^{-2t}z|
        <
        e^{-2t}\rho(|w|)
        \leq
        e^{-t}\rho(|w|)
        \leq
        \rho(e^t|w|).
\]
Thus
\[
        e^{tA}(z,w)\in\Omega
        \qquad \text{for all } t\geq 0,
\]
and so
\[
        e^{tA}(\Omega)\subset\Omega
        \qquad \text{for all } t\geq 0.
\]

The projections associated
with $ \mathbb C^2=E^s\oplus E^{cu}$ are the following:

\[
        P^s(z,w)=(z,0),
        \qquad
        P^{cu}(z,w)=(0,w).
\]
Let \(K\subset\mathbb C^2\) be compact,   $ M_K:=\sup_{(z,w)\in K}|z|$ and $N_K:=\sup_{(z,w)\in K}|w|.$ Then
\[
        \sup_{(z,w)\in K}\|e^{tA}P^s(z,w)\|
        =
        M_K e^{-2t}.
        \tag{1}
\]
\noindent
We need a lower bound for the distance from \(e^{tA}P^{cu}(K)\) to
\(\partial\Omega\). Since \(m(r)\geq 1\) for all \(r\geq 0\), we have
\[
        \rho(r)\geq \frac{1}{1+r}.
        \tag{2}
\]
The boundary of \(\Omega\) is
\[
        \partial\Omega
        =
        \{(\zeta,\omega)\in\mathbb C^2:|\zeta|=\rho(|\omega|)\}.
\]
We claim that there exists a constant \(c>0\) such that
\[
        \operatorname{dist}\bigl((0,W),\partial\Omega\bigr)
        \geq
        \frac{c}{1+|W|}
        \qquad \text{for all } W\in\mathbb C.
        \tag{3}
\]
Indeed, let \((\zeta,\omega)\in\partial\Omega\). Then
\[
        |\zeta|=\rho(|\omega|)\geq \frac{1}{1+|\omega|}.
\]
If
\[
        |\omega-W|\geq \frac{1}{2(1+|W|)},
\]
then
\[
        \|(\zeta,\omega)-(0,W)\|
        \geq
        |\omega-W|
        \geq
        \frac{1}{2(1+|W|)}.
\]
On the other hand, if
\[
        |\omega-W|<\frac{1}{2(1+|W|)},
\]
then in particular \(|\omega-W|<1/2\), and hence
\[
        1+|\omega|
        \leq
        1+|W|+|\omega-W|
        \leq
        \frac32(1+|W|).
\]
Therefore
\[
        |\zeta|
        \geq
        \frac{1}{1+|\omega|}
        \geq
        \frac{2}{3}\frac{1}{1+|W|}.
\]
Combining the two cases, we obtain \((3)\), for example with
\(c=1/2\).

Since $e^{tA}P^{cu}(z,w)=(0,e^t w)$,
it follows from \((3)\) that
\[
\begin{aligned}
        \operatorname{dist}
        \bigl(e^{tA}P^{cu}(K),\partial\Omega\bigr)
        &=
        \inf_{(z,w)\in K}
        \operatorname{dist}
        \bigl((0,e^t w),\partial\Omega\bigr)       \\
        &\geq
        \inf_{(z,w)\in K}
        \frac{c}{1+e^t|w|}                         
        \geq
        \frac{c}{1+e^t N_K}.
\end{aligned}
        \tag{4}
\]
Since
\[
        M_K e^{-2t}\bigl(1+e^t N_K\bigr)
        =
        M_K e^{-2t}+M_KN_K e^{-t}
        \longrightarrow 0
\]
as \(t\to\infty\), there exists \(T_K>0\) such that
\[
        M_K e^{-2t}
        <
        \frac{c}{1+e^t N_K}
        \qquad \text{for all } t\geq T_K.
\]
Therefore
\[
        \operatorname{dist}
        \bigl(e^{tA}P^{cu}(K),\partial\Omega\bigr)
        >
        \sup_{(z,w)\in K}\|e^{tA}P^s(z,w)\|
        \qquad \text{for all } t\geq T_K.
\]
Hence,
by Theorem~\ref{T:orbitwise-thickness}, \(\Omega\) is Runge.
This example is not covered by Corollary~\ref{T:Runge}. Indeed, for \(R\geq 2\),
\[
        \rho(R)=\frac{10}{1+R},
\]
and the point
$(\rho(R),R)\in \partial\Omega$. Since \((0,R)\in E^u\), we get that
\[
        \operatorname{dist}\bigl((0,R),\partial\Omega\bigr)
        \leq
        \rho(R)
        =
        \frac{10}{1+R}
        \longrightarrow 0
        \quad \text{as } R\to\infty.
\]
Hence, $\operatorname{dist}(E^u,\partial\Omega)=0.$
\end{example}

 The following example shows that the distance condition in Theorem~\ref{T:orbitwise-thickness} can not be removed.

\begin{example}\label{E:Example0}
Let $\Omega:=\mathbb C^2\setminus\{(z,w)\in\mathbb C^2:zw=1\}$.
Then \(\Omega\) is not Runge. Indeed, let
$K:=\{(z,w)\in\mathbb C^2: |z|=0.6,\ w=2\}$.
The polynomial hull of \(K\) is $\widehat K
        =
        \{(z,w)\in\mathbb C^2: |z|\leq 0.6,\ w=2\}.$
But
$(1/2,2)\in \widehat K$ and
        $(1/2,2)\notin\Omega.$
Thus \(\Omega\) cannot be Runge.
On the other hand, \(\Omega\) is positive-time flow invariant with respect to
\[
        A=
        \begin{pmatrix}
        1 & 0\\
        0 & -1
        \end{pmatrix},
\]
because the term \(zw\) is preserved by the flow
$(z,w)\mapsto (e^t z,e^{-t}w).$
We also get that
\[
        E^{cu}=E^u=\{(z,0):z\in\mathbb C\}\subset\Omega.
\] 
Clearly, 
\[\operatorname{dist}
        \bigl(e^{tA}P^{cu}(\cplx^2),\partial\Omega\bigr)=0\quad \text{and}\quad \sup_{z\in B(0,R)}\|e^{tA}P^s z\|=R.
        \]
        Thus, \(\Omega\) satisfies all the hypotheses of
Theorem~\ref{T:orbitwise-thickness} except the distance condition. This shows that the
assumption
\[
       \operatorname{dist}\bigl(e^{tA}P^{cu}(K),\partial\Omega\bigr) > \sup_{z\in K}\|e^{tA}P^s z\|
\]
for every compacts in $\cplx^2$ cannot be omitted.
\end{example}

\section{Linearization  of the  Vector field}

In this section, we present the proof of Theorem~\ref{T:integrability-linearization}.

\begin{proof}[Proof of Theorem~\ref{T:integrability-linearization}] We break the proof into four steps.
\smallskip

\noindent  \textbf{Step I: Construction of the limiting map}
 \smallskip
 
\noindent Let $X(t,z)$ be the flow of the vector field $V$. Hence, 
\[
\frac{dX(t,z)}{dt}=V(X(t,z)) \forall t \geq 0, \forall z \in \cn. 
\]
 We  now deduce: 
    \begin{align}\label{E:E1}
    & e^{-tA}\frac{dX(t,z)}{dt}=e^{-tA}A(X(t,z))+e^{-tA}R(X(t,z))\notag\\ 
\impl & \frac{d(e^{-tA}X(t,z))}{dt}=e^{-tA}R(X(t,z))\notag\\ 
\impl e^{-tA}X(t,z)&=z+\int_{0}^{t}e^{-sA}R(X(s,z))\,ds.  
\end{align}
By the assumed locally uniform convergence of the improper integral,
the right-hand side converges locally uniformly on $\mathbb C^n$ as
$t\to\infty$. Thus
\[
        F(z):=
        z+
        \int_0^\infty e^{-sA}R(X_s(z))\,ds
\]
is well defined, and
\[
        e^{-tA}X_t(z)\longrightarrow F(z)
\]
locally uniformly on \(\mathbb C^n\).

For each $t\geq 0$, define
\[
        f_t:=e^{-tA}\circ X_t .
\]
Since $X_t$ is an automorphism of $\mathbb C^n$ with inverse
$X_{-t}$, and since $e^{-tA}$ is a linear automorphism, each
$f_t\in \operatorname{Aut}(\mathbb C^n)$. Since
$f_t\to F$ locally uniformly and each $f_t$ is holomorphic, the
Weierstrass theorem implies that $F$ is holomorphic.

\medskip

\noindent \textbf{Step II: Injectivity of the limiting map.} 
\smallskip

\noindent We have $f_t:=e^{-tA}\circ X_{t} \in \text{Aut}(\cn)$ and $D(e^{-tA}X_{t})(0)=e^{-tA}DX_{t}(0)=e^{-tA}e^{tA}=I_{n} $ for all $t\geq 0$. Since $ Df_{t}(0) \to DF(0)$ as $t \to \infty$, hence, we obtain that $DF(0)=I_{n}$. Since $\text{det}(D(e^{-tA}X_{t})(z)) \neq 0$ for all $z \in \cn$, therefore,  by Hurwitz theorem in several variables, we conclude that $\text{det}(DF(z))$ is either non-zero on $\Om$ or identically zero on $\Om$. But, $\text{det}(DF(0))=1$ implies that $\det(DF(z)) \neq 0$ for all $z \in \Om$. Therefore, by invoking inverse function theorem, we obtain that $F$ is a local biholomorphism.
We need to show that $F$ is injective. Suppose not. Then, there exist two distinct points $z_{1}, z_{2}$ in $\Om$ such that $F(z_{1})=F(z_{2})$. Since $F$ is locally injective, hence, there exists $r>0$ such that $F$ is injective on $B(z_1,r)$. Shrinking $r>0$, if needed, we assume that $z_{2} \notin \overline{B(z_1,r)}$. Define $\widetilde{F}_{t}:=f_{t}-f_{t}(z_{2})$
 and $\widetilde{F}:=F-F(z_{2})$. Clearly, $\widetilde{F}^{-1}(0) \cap \partial B(z_1,r) =\emptyset$. 
 Choose $\alpha>0$ such that $\inf_{\bdy B(z_1,r)}|\widetilde{F}|>\alpha$.
 
 Now consider $g_t:=\widetilde{F}_{t}-\widetilde{F}$. We have
 $$
     \sup_{\bdy B(z_1,r)}|g_t(z)|\leq \sup_{\bdy B(z_1,r)}\left ( |f_t(z)-F(z)|+|f_t(z_2)-F(z_2)|\right). 
 $$
 Hence, there exists $N_0\in\natu$ such that $\sup_{\bdy B(z_1,r)}|g_{t}(z)|\leq \al$ for all $t>N_0$. Hence, we obtain that 
 $$
 \sup_{\bdy B(z_1,r)}|g_t(z)|\leq \al < |\widetilde{F}(z)| \;\;\forall z\in\bdy B(z_1,r).
 $$
Therefore, we infer from Result~\ref{R:Lloyd}, that $\widetilde{F}$ and $\widetilde{F_t}$ has same number of zeros on $B(z_1,r)$ for all $t>N_0$. This contradicts the fact that $f_{t}$ are injective. Therefore, $F:\Om \to \cn$ is an injective holomorphic map.
\\
\medskip
\\
\textbf{Step III: The limiting map conjugates the nonlinear flow $X_{t}(z)$ to the linear flow $e^{tA}(z)$.}
\smallskip

\noindent
We have that for every compact set $K \subset \cn$ the map $e^{-tA}X_{t} \to F(z)$ uniformly over $K$ as $t \to \infty$. Let $s > 0$
 and $\eps>0$ be any given real number. Let $r>0$ such that $(k_{+}(A)+r)<0$. Then from the relation given in  Result~\ref{L:exp}, we get that there exists $m(r)>0$ such that 
 \begin{align}\label{E:le0}
     \|e^{tA}z\| &\leq m(r)e^{(k_{+}(A)+r)t}\|z\|~~~\forall t \geq 0.
 \end{align}
 Now choose $T_{s}>0$ such that the following holds for all $z \in K$ for all $t\geq T_{s}$:
 \begin{align}
     \|F(X_{s}(z))-e^{-tA}\circ X_{t}(X_{s}(z))\|&<\frac{\eps}{2}\label{E:le1}\\
    \|e^{-tA}X_{t}(z)-F(z)\|&<\frac{\eps}{2m(r)}\label{E:le2}. 
 \end{align}
 We now have the following 
 \begin{align}
    & \|F(X_{s}(z))-e^{sA}F(z)\|\notag\\
    &\leq  \|F(X_{s}(z))-e^{-T_{s}A}X_{T_{s}}(X_{s}(z))\|+\|e^{-T_{s}A}X_{T_{s}}(X_{s}(z)-e^{sA}F(z)\|\notag\\
     &\leq \|F(X_{s}(z))-e^{-T_{s}A}X_{T_{s}}(X_{s}(z))\|+\|e^{sA}\big(e^{-(T_{s}+s)A}X_{T_{s}}(X_{s}(z)-F(z)\big)\notag\|,
     \end{align}
     using \eqref{E:le0} we obtain the following from above equation
     \begin{align}\label{E:le3}
     \|F(X_{s}(z))-e^{sA}F(z)\| &\leq \|F(X_{s}(z))-e^{-T_{s}A}X_{T_{s}}(X_{s}(z))\|\notag\\
 &+m(r)e^{s((k_{+}(A)+r))}\|\big(e^{-(T_{s}+s)A}X_{s+T_{s}}(z)-F(z)\big)\|.
\end{align}
From \eqref{E:le1},  \eqref{E:le2}, \eqref{E:le3} we obtain that 
\begin{align}
    \|F(X_{s}(z))-e^{sA}F(z)\|&< \eps.
\end{align}
Since $\eps>0$ is an arbitrary real number, we conclude that $F(X_{s}(z))=e^{sA}F(z)$. Consequently, we get that $F\circ X_{s}\circ F^{-1}(z)=e^{sA}(z)$ for all $s \geq 0$ and $z \in \cn$.
\medskip

\noindent
{\textbf{Step IV:  Surjectivity of $F$.}} 
\smallskip

\noindent It is enough that $F$ is onto. It follows from Step III that  
\[
F(z)=e^{-sA}\circ F \circ X_{s}(z)\;\; \forall z \in \cn\;\; \forall s\geq 0.
\]
Let $w \in \cn$ and $s>0$ large enough such that $e^{sA}w \in F(\cn)$. Then 
\[
F(X_{-s}(F^{-1}(e^{sA}w)))=w.
\]
Therefore, $F$ is surjective.

We obtain that the limiting map $F\in\operatorname{Aut}(\cplx^n)$.
The relation
\[
        F\circ X_s=e^{sA}\circ F
\]
has already been proved for \(s\geq 0\). Since both \(X_s\) and
\(e^{sA}\) are one-parameter groups, the same identity holds for all
\(s\in\mathbb R\). Equivalently,
\[
        F\circ X_s\circ F^{-1}=e^{sA}
        \qquad \text{for all } s\in\mathbb R.
\]

If the vector field is complete in complex time, then for each fixed
\(z\in\mathbb C^n\), the maps
\[
        \tau\longmapsto F(X_\tau(z))
        \qquad \text{and} \qquad
        \tau\longmapsto e^{\tau A}F(z)
\]
are entire maps from \(\mathbb C\) to \(\mathbb C^n\). They agree for
all real \(\tau\), and hence, by the identity theorem, they agree for
all \(\tau\in\mathbb C\). Therefore, we conclude that $F \in {\rm Aut}(\cn)$ linearizes the action of the group  $(\cplx, +)$  on $\cn$ defined by the flow of the complete vector field $V$.

  \end{proof}

\begin{proof}[Proof of Corollary~\ref{L:composconver}]

 By assumption, the vector field has the form $V(z)=Az+R(z)$, where $R(z)=O(|z|^{m})$.
 Let $\delta>0$ small enough such that $|R(w)|\leq C|w|^{m}$ for all $w\in B(0, \delta)$.
 Since $X(t,z) \to 0$ as $t \to \infty $ uniformly over every compact subsets of $\cn$, for a given compact subset $K \subset \cn$ there exists a real number $m_{K}>0$ such that $X(t,z) \in B(0, \delta)$ for all $z \in K$ for all $t>m_{K}$. Let $K \subset \cn$ be a compact set containing the point $z$.
Using Result~\ref{L:exp} and the assumption on the vector field we obtain that
\begin{align}\label{E:E2}
\bigg|\int_{m_{K}}^{\infty}e^{-sA}R(X(s,z))\,ds\bigg|&\leq C\int_{m_{K}}^{\infty}e^{-s(k_{-}(A)-\eps)}|X(s,z)|^{m}\,ds.  
 \end{align}
   Since all the eigenvalues of the matrix $A$ have strictly negative real parts, hence, by Hartman-Grobman theorem, we get that there exists a neighborhood $U$ of the origin and a  $\smoo^{1}$ smooth map $H:U \to H(U)$ such that 
\[
X(t,z)=H\circ e^{tA}\circ H^{-1}(z)\;\;\forall t \geq 0, \forall z \in U.
\]
Here the origin is the globally attractive equilibrium point of the vector field $V$.  Therefore, the map $H$ can be extended to the whole of $\cn$ in the following way: Choose a large  $s>0$ such that $X(s,z) \in U$ and $e^{sA}z\in U$. We define 
$\widetilde{H}:\cn \to \cn$  by $\widetilde{H}(z):=e^{-sA}\circ H \circ X_{s}(z)$. 
It is easy to see that the map $\widetilde{H}$ is independent of $s$. Let $s_{1}, s_{2}>0$ such that $X_{s_{1}}(z), X_{s_{2}}(z) \in U$. Assume that $s_{2}>s_{1}$. We get that 
\begin{align*}
    e^{-s_{2}A}H(X_{s_{2}}(z))&=e^{-s_{2}A}H(X_{s_{2}-s_{1}}(X_{s_{1}}(z)))\\
\end{align*}
Since $X_{s_{1}}(z) \in U$ and $H$ conjugate the flow to its linearization on $U$, hence, from the above equation we get that
\begin{align*}
    e^{-s_{2}A}H(X_{s_{2}}(z))&=e^{-s_{2}A}e^{(s_{2}-s_{1})A}H(X_{s_{1}}(z))=e^{-s_{1}A}H(X_{s_{1}}(z)).
\end{align*}
We now obtain that 
\begin{align}
    \widetilde{H}\circ X_{t}\circ \widetilde{H}^{-1}(z)&=e^{-sA}\circ H \circ X_{s}(X_{t}\circ X_{-s} \circ \widetilde{H}^{-1}(e^{sA}(z))\notag\\
    &=e^{-sA}\circ H \circ X_{t}\circ H^{-1}\circ(e^{sA}z)\notag\\
 &=e^{tA}z. \label{E:E3}
\end{align}
Therefore, from \eqref{E:E2} and \eqref{E:E3} we obtain that 
\begin{align}
\bigg|\int_{m_{K}}^{\infty}e^{-sA}R(X(s,z))\,ds\bigg|&\leq C\int_{m_{K}}^{\infty}e^{-s(k_{-}(A)-\eps)}|\widetilde{H}^{-1}e^{tA}\widetilde{H}(z)|^{m}\,ds, \notag
\intertext{since $\widetilde{H}$ is $\smoo^{1}$-smooth, hence, using mean value   inequality we obtain that}
\bigg|\int_{m_{K}}^{\infty}e^{-sA}R(X(s,z))\,ds\bigg|&\leq\widetilde{C}\int_{m_{K}}^{\infty}e^{-s(k_{-}(A)-\eps)}e^{ms(K_{+}(A)+\eps)}\,ds\notag\\
&=\widetilde{C}\int_{m_{K}}^{\infty}e^{s(-k_{-}(A)+m.k_{+}(A)+(m-1)\eps)}\,ds. \label{eq-estimate}
\end{align}
By our assumption, there exists $\eps>0$ such that 
\[
s(-k_{-}(A)+m.k_{+}(A)+(m-1)\eps)<0.
\]
Hence, the right hand side of \eqref{eq-estimate} converges.
Using Theorem~\ref{T:integrability-linearization} we get the linearization.
\end{proof}
 
We now give a couple of examples where the linearizing map can be explicitly computed.

 \begin{example}
    Consider a holomorphic  vector field on $\CC$ defined by $V(z_{1}, z_{2})=(-z_{1},-2z_{2}+z_{1}^{4})$. The flow is $X(t,(z_{1},z_{2}))=(e^{-t}z_{1}, e^{-2t}(z_{2}+\frac{z_{1}^{4}}{2}(1-e^{-2t})))$. 
Here   $V(z)=Az+O(|z|^{4})$, where $A=\begin{bmatrix}
    -1 & 0\\
    0 & -2
\end{bmatrix}$ and 
\[
        4\max\{-1,-2\}
        =
        -4
        <
        -2
        =
        \min\{-1,-2\}.
\]
Thus the spectral-gap condition in Corollary~\ref{L:composconver} is satisfied.
Although the spectrum has a resonance, by Corollary~\ref{L:composconver}, there exists an automorphism
\(F\in\operatorname{Aut}(\mathbb C^2)\) such that
\[
        F\circ X_t\circ F^{-1}=e^{tA}.
\]
The automorphism is
\[
  F(z_1,z_2)=\left(z_1,z_2+\frac{z_1^4}{2}\right).
\]

This conjugacy allows us to study the nonlinear flow \(X_t\) by passing
to the coordinates
\[
        w_1=z_1,
        \qquad
        w_2=z_2+\frac{z_1^4}{2},
\]
in which the flow becomes
\[
        w_1(t)=e^{-t}w_1,
        \qquad
        w_2(t)=e^{-2t}w_2.
\]
As a first consequence, we obtain invariant curves for the nonlinear
flow. For the linear flow
$t\mapsto(e^{-t}w_1,e^{-2t}w_2)$,
the meromorphic function
$J(w_1,w_2):=w_2/w_1^2$
is a first integral on the set \(\{w_1\neq 0\}\).
By Corollary~\ref{coro-transfer}, we obtain a first integral for the nonlinear
flow:
\[
        I(z_1,z_2)
        =
        J(F(z_1,z_2))
        =
        \frac{z_2+\frac{z_1^4}{2}}{z_1^2}.
\]
Consequently, for each $c\in\cplx$ the curves
\[
        \Gamma_c
        :=
        \left\{
        (z_1,z_2)\in\mathbb C^2:
        z_2+\frac{z_1^4}{2}=c z_1^2
        \right\},
        \]
are invariant under the nonlinear flow.  
\end{example}

\begin{example}
Consider
$V(z_1,z_2)=(-z_1,-2z_2+z_1z_2)$
on \(\mathbb C^2\). Then $V(z)=Az+R(z)$, where
\[
        A=
        \begin{pmatrix}
        -1 & 0\\
        0 & -2
        \end{pmatrix},
        \qquad
        R(z)=(0,z_1z_2).
\]
Thus, \(R(z)=O(\|z\|^2)\), but
\[
        m\max_{\lambda\in\sigma(A)}\operatorname{Re}\lambda
        =
        2(-1)
        =
        -2 =
        \min_{\lambda\in\sigma(A)}\operatorname{Re}\lambda
\]
Thus, the strict spectral-gap condition
fails.
The flow is
\[
        X_t(z_1,z_2)
        =
        \left(
        e^{-t}z_1,\,
        z_2\exp(-2t+z_1(1-e^{-t}))
        \right).
\]
Therefore,
\[
        e^{-sA}R(X_s(z))
        =
        \left(
        0,\,
        z_1z_2
        \exp(-s+z_1(1-e^{-s}))
        \right),
\]
which is integrable on compact subsets of \(\mathbb C^2\).
Hence, Theorem~\ref{T:integrability-linearization} applies. The limiting automorphism is
\[
        F(z_1,z_2)=\lim_{t\to\infty}e^{-tA}X_t(z_1,z_2)
        =
        (z_1,e^{z_1}z_2),
\]
and 
$F\circ X_t\circ F^{-1}=e^{tA}$.
\end{example}
		
\noindent {\bf Acknowledgements.} 
Sanjoy Chatterjee is supported by the Institute Postdoctoral Fellowship of the Indian Institute of Technology, Kanpur.   Sushil Gorai is partially supported by a Core Research Grant (CRG/2022/003560) of ANRF (erstwhile SERB), Dept. of Science and Technology, Govt. of India.



\end{document}